\documentclass[11pt]{amsart}
\usepackage{amsmath,amssymb,enumerate,etoolbox,mathrsfs}
\usepackage[all]{xy}
\newtheorem{thm}{Theorem}

\newtheorem{prop}[thm]{Proposition}

\newtheorem{lem}[thm]{Lemma}
\theoremstyle{definition}

\newtheorem*{remark*}{Remark}

\newcommand{\At}{\tilde{A}}
\newcommand{\A}{A}
\newcommand{\Bt}{\tilde{B}}
\newcommand{\B}{B}
\newcommand{\X}{X}
\renewcommand{\P}{\mathcal{P}}

\newcommand{\E}{\mathcal{E}}
\newcommand{\C}{\mathbb{C}}
\newcommand{\F}{\mathbb{F}}

\newcommand{\Q}{\mathbb{Q}}
\newcommand{\Z}{\mathbb{Z}}
\newcommand{\SL}{{\rm SL}}
\newcommand{\PSL}{{\rm PSL}}
\newcommand{\PGL}{{\rm PGL}}
\newcommand{\PGSp}{{\rm PGSp}}
\newcommand{\Spin}{{\rm Spin}}
\newcommand{\SU}{{\rm SU}}
\newcommand{\PSU}{{\rm PSU}}
\newcommand{\PGU}{{\rm PGU}}
\newcommand{\Sp}{{\rm Sp}}
\newcommand{\PSp}{{\rm PSp}}
\newcommand{\SO}{{\rm SO}}
\newcommand{\PO}{{\rm PO}}
\newcommand{\GL}{{\rm GL}}

\newcommand{\Res}{{\rm Res}}
\newcommand{\Aut}{{\rm Aut}}
\newcommand{\Gal}{{\rm Gal}}
\newcommand{\Tr}{{\rm Tr}}

\newcommand{\ord}{{\rm ord}}

\newcommand{\Irr}{{\rm Irr}}

\newcommand{\ad}{{\rm ad}}
\renewcommand{\sc}{{\rm sc}}

\renewcommand{\O}{{\mathcal{O}}}
\newcommand{\U}{\mathcal U}
\renewcommand{\L}{\mathcal L}
\newcommand{\uG}{\underline{G}}
\newcommand{\uH}{\underline{H}}
\newcommand{\uS}{\underline{S}}
\newcommand{\uT}{\underline{T}}
\newcommand{\uZ}{\underline{Z}}

\title[Sparse character tables in high rank]{The sparsity of character tables of high rank groups of Lie type}
\author[M.~J.~Larsen]{Michael J.~Larsen}
\address{Department of Mathematics, Indiana University, Bloomington, IN, USA}
\email{mjlarsen@indiana.edu}
\author[A.~R.~Miller]{Alexander~R.~Miller}
\address{Faculty of Mathematics, University of Vienna, Austria}
\email{alexander.r.miller@univie.ac.at}
\begin{document}
\begin{abstract}
In the high rank limit, the fraction of non-zero character table entries of finite simple groups of Lie type goes to zero.
\end{abstract}
\thanks{ML was partially supported by the NSF grant DMS-1702152. AM was partially supported by the Austrian Science Foundation.}
\subjclass[2010]{20C33}
\maketitle
\thispagestyle{empty}
\section{Introduction}

Let $G$ be a finite group.  
Let $G^G$ denote the set of conjugacy classes of $G$ and $\Irr(G)$ the set of irreducible characters, so
$k(G) = |G^G| = \Irr(G)$.
A well-known theorem of Burnside asserts that if $\chi\in \Irr(G)$ and $\chi(1)>1$, then there exists $g\in G$ such that $\chi(g) = 0$; in particular,
there are zero entries in the character table of every non-abelian $G$.
In fact, one can make much stronger statements about the subset of $G^G\times \Irr(G)$ determined by the vanishing condition,
and there is a substantial literature devoted to such results.  We are interested in the opposite extreme from abelian groups,
namely groups for which almost all entries are zero.

We define the \emph{sparsity} $\Sigma(G)$ to be the fraction of non-zero entries in the character table of $G$:
$$\Sigma(G) := \frac{|\{(g^G, \chi)\in  G^G\times\Irr(G)\mid \chi(g) \neq 0\}|}{k(G)|\Irr(G)|}.$$
For finite simple groups of bounded rank, it is not too difficult to analyze the asymptotic behavior of $\Sigma(G)$.
For instance,
$$\lim_{q\to \infty}\Sigma(L_2(q)) = \frac 12.$$
For other series of finite simple groups with fixed Dynkin diagrams, the limit is rational and non-zero.
In this paper, we consider what happens when $G$ ranges over finite simple groups of Lie type of unbounded rank.  Our result is the following.

\begin{thm}
\label{main}
Given any sequence $G_i$ of finite simple groups of Lie type with rank tending to $\infty$,  $\lim_{i\to\infty} \Sigma(G_i) = 0$.
\end{thm}

To round out the story, it would be interesting to know whether $\Sigma(A_n)\to 0$ (or, equivalently, whether $\Sigma(S_n)\to 0$) as $n\to \infty$.
This remains open.   The numerical evidence \cite{ARMiller2}
seems to point to a limit strictly between $0$ and $1$.
Interestingly, Miller proved \cite{ARMiller1}
that for random pairs $(\chi,g)\in \Irr(G)\times G$, the probability that $\chi(g)=0$ goes to $1$ as $G$ ranges over symmetric groups.

The proof of Theorem~\ref{main}
uses a trick of Burnside and roughly parallels that of \cite{GLM}.  Given a pair $(g^G,\chi)$, let
$$d_{g^G,\chi} := \frac{\chi(1)}{(\chi(1),|g^G|)}.$$
By \cite{GLM}, $\chi(g)$ is divisible by $d_{g^G,\chi}$ in a ring of cyclotomic integers. 
If $\alpha\neq 0$ is an algebraic integer, then the average of $|\alpha_i|^2$ over all conjugates $\alpha_i$ of $\alpha$
is at least $1$ \cite[p.~459]{PXGallagher2}. 
Therefore, if $\alpha$ is divisible by a rational integer $d$, then the average is at least $d^2$.  The multiset of values $\chi(g)\neq 0$ where $\chi(g)$ is divisible by some rational integer $d > D$
is stable under the action of $\Gal(\Q(\zeta_{|G|})/\Q)$, so the average of $|\chi(g)|^2$ over all such values is greater than $D^2$.
By the orthonormality of characters, the average of $|\chi(g)|^2$ over all pairs $(\chi,g)$ is $1$.  Of course, this average is over elements rather than conjugacy classes,
so a key ingredient of the argument is Proposition~\ref{Few large centralizers},
which implies that, for finite simple groups of Lie type, it does not make much difference 
which kind of average one takes.
This is not true for alternating groups, since the partition associated to a randomly chosen element of $S_n$ has $\log n +o(\log n)$ parts \cite{ET}, 
while a typical partition of $n$ has $(\pi^{-1}\sqrt{3/2} + o(1)) \sqrt n\log n$ parts \cite{EL}.

To show that $d_{g^G,\chi}$ is usually large, we  show that for most choices of $(g^G,\chi)$, we can find a large Zsigmondy prime $\ell$ such 
$\ord_\ell |g^G| < \ord_\ell \chi(1)$.  To do this, we need to have a good qualitative understanding of the degrees of irreducible characters of $G$.
This is provided by the Lusztig theory.  In the large $q$ limit, regular semisimple elements and irreducible Deligne-Lusztig characters predominate, and it is
relatively easy prove the needed estimates.  For fixed $q$, we have to work harder, but most of what is needed is already available in work of Fulman-Guralnick \cite{FG,FG2}
and Larsen-Shalev \cite{LS}.

\section{General framework}
A key technical difficulty in implementing our strategy is that in counting conjugacy classes and characters, it is often easier to work not with $G$ but with
some closely related group.  For instance, for $\PSL_n(\F_q)$, it is easier to study conjugacy classes of $\SL_n(\F_q)$ and characters of $\PGL_n(\F_q)$.
To deal with these difficulties, we 
consider the following general situation.  Suppose that we have maps of finite sets
\begin{equation}
\label{data}
\xymatrix{\At\ar^f[r]\ar@{->>}_\phi[d]&\P&\Bt\ar[l]_g\ar@{->>}[d]^\psi\\
\A&&\B}
\end{equation}
and a subset $\X\subset \A\times \B$.  Our goal is to show that, under suitable conditions, $|\X|$ is small compared to $|\A\times \B|$.  We say $\P^\circ \subset \P$ and $\A^\circ\subset \A$ are \emph{compatible} if
$f^{-1}(\P^\circ) = \phi^{-1}(\A^\circ)$, and likewise for $\P^\circ$ and $\B^\circ$.
In this case, we define $\At^\circ\subset \At$ (respectively $\Bt^\circ\subset \Bt$) to be this common
inverse image.

In the intended application, $\A$ will be the set of conjugacy classes of a finite simple group $G$, $\B$ the set of irreducible characters of $G$, and  $\X$, which we want to prove small, will be the set of pairs $(g^G,\chi)$ for which $\chi(g)\neq 0$.
The precise definitions of $\At$, $\Bt$, and $\P$ are given in \S5.  The definitions of the sets $\A^\circ,\At^\circ,\B^\circ,\Bt^\circ,\P^\circ$ depend on a single parameter as explained in \S9.

\begin{prop}
\label{Master}
There exists an absolute constant $N$ with the following property.
For all $\epsilon > 0$  there exists $\delta>0$ such that given data \eqref{data}, a subset $\X\subset \A\times \B$, and 
compatible subsets $\A^\circ\subset \A$, $\B^\circ\subset \B$, and 
$\P^\circ\subset \P$ which are compatible and
satisfy the following conditions, then $|\X| \le \epsilon |\A\times \B|$:
\begin{enumerate}
\item $|\A^\circ| > (1-\delta)|\A|$.
\item $|\B^\circ| > (1-\delta)|\B|$.
\item For $a_1\in \A^\circ$, $a_2\in\A$ we have $|\phi^{-1}(a_1)|\ge |\phi^{-1}(a_2)|$.
\item For $b_1\in \B^\circ$, $b_2\in \B$ we have $|\psi^{-1}(b_1)|\ge |\psi^{-1}(b_2)|$.
\item For all $n\ge N$, $\{P\in \P\mid |f^{-1}(P)|\ge n\}$ has less than $n^{-2} |\P|$ elements.
\item For all $n\ge N$, $\{P\in \P\mid |g^{-1}(P)|\ge n\}$ has less than $n^{-2} |\P|$ elements.
\item $|\P|\le  |\At|$.
\item $|\P|\le |\Bt|$.
\item The set of pairs $(P_1,P_2)\in \P^\circ\times \P^\circ$ such that $(P_1,P_2)\in (f,g)((\phi,\psi)^{-1}(\X))$ has cardinality less than $\delta |\P|^2$.
\end{enumerate}
\end{prop}

\begin{proof}
By conditions (3) and (4),
every $(\phi,\psi)$-fiber over $\A^\circ\times \B^\circ$ has cardinality at least 
$$\frac{ |\At \times \Bt|}{|\A\times \B|}.$$
Let $x=|X|/|\A\times\B|$. 
Then by conditions (1) and (2),
$$|\X\cap (\A^\circ\times \B^\circ)| > (x-2\delta)|\A\times \B|,$$
so
\begin{equation}
\label{inverse image bound}
|(\phi,\psi)^{-1}(X) \cap (\At^\circ\times \Bt^\circ)| = |(\phi,\psi)^{-1}(\X\cap  (\A^\circ\times \B^\circ))| > (x-2\delta)|\At\times \Bt|.
\end{equation}

Let $\P_{f,n}$ denote the set of elements  $P\in \P^\circ$  such that $|f^{-1}(P)| = n$.  Let $M\geq N$.  By condition (5), 
\begin{equation}
\label{Few A-tildes in big fibers}
\begin{split}
\bigm| \bigcup_{n\ge M}& f^{-1}(\P_{f,n}) \bigm| = \sum_{n\ge M} n |\P_{f,n}| \\
                                                                     &= M |\{P\in \P^\circ\mid |f^{-1}(P)|\ge M\}| + \sum_{i=1}^\infty |\{P\in \P^\circ\mid |f^{-1}(P)|\ge M+i\}|\\
                                                                     & \le \frac{|\P|}M + \sum_{i=1}^\infty \frac{|\P|}{(M+i)^2} \le \frac{2|\P|}{M}.
\end{split}
\end{equation}
Likewise, condition (6) implies
$$\bigm| \bigcup_{n\ge M} g^{-1}(\P_{g,n}) \bigm| \le \frac{2 |\P|}M.$$
By inequality \eqref{inverse image bound} and conditions (7) and (8),
the subset of $(\phi,\psi)^{-1}(\X)$ consisting of elements which map by $(f,g)$ to 
$$\{P\in \P^\circ\mid |f^{-1}(P) |< M\}\times \{P\in \P^\circ\mid |g^{-1}(P)| < M\}$$
has more than 
$$(x-2\delta-4M^{-1})|\P|^2$$
elements.  
On this set the map $(f,g)$ takes at most $M^2$ elements to any given element,
so the cardinality of the image by $(f,g)$ is at least
$$\frac{(x-2\delta-4M^{-1})|\P|^2}{M^2}.$$
By condition (9), this must be less than $\delta|\P|^2$, so $x< M^2\delta+2\delta+4M^{-1}$.
By choosing $M$ larger than $N$ and $8/\epsilon$, we get $x<\epsilon$ if 
$\delta<\frac{\epsilon}{2M^2+4}$. 
\end{proof}

\section{Subexponential sequences}

In this section, we prove some basic facts about subexponential sequences that will be useful for checking the hypotheses of Proposition~\ref{Master}.

We say that a sequence $a_1,a_2,\ldots$ of nonnegative integers is \emph{subexponential} if for all $\gamma > 1$ we have $\lim_{n\to \infty} \gamma^{-n}a_n = 0$.  This is equivalent to the condition that $\sum_n a_n z^n$ converges in the open unit disk.  It is clear from this criterion
that the coefficients of the product of power series with subexponential coefficients again has subexponential coefficients.
From the definition, it is clear that the termwise product of subexponential sequences is again subexponential.

\begin{lem}
\label{subexp}
If $(a_i)_{i=1,2,\ldots}$ is subexponential, then the sequence $A_m$ of coefficients of the formal power series 
$$A(z) := \prod_{i=1}^\infty (1-z^i)^{-a_i} = \sum_m A_m z^m$$
is likewise subexponential.
\end{lem}

\begin{proof}
As $a_i\ge 0$ for all $i$, 
$$A^{(n)}(z):=\prod_{i=1}^n (1-z^i)^{-a_i} = \sum_{m=0}^\infty A^{(n)}_m z^m$$
has nonnegative coefficients, and for each $m\ge 0$, the sequence $(A^{(n)}_m)_{n=1,2,\ldots}$ is nondecreasing.  Therefore, for any $z_0\in (0,1)$, the sequence $(A^{(n)}(z_0))_{n=1,2,\ldots}$ is nondecreasing.  As $A_m = A^{(n)}_m$, for all $n\ge m$, we have
$$A^{(n)}(z_0)\ge \sum_{m=0}^n A_m z_0^m.$$
Thus, if $A(z)$ converges at $z=z_0$, then
$\lim_{n\to \infty} A^{(n)}(z_0)$ exists and equals $A(z_0)$, and conversely, if $\lim_{n\to \infty} A^{(n)}(z_0)$ exists, then its limit is an upper bound for the increasing sequence of partial sums of $\sum_{m=0}^\infty A_m z_0^m$, so this series converges.

As $\lim_{z\downarrow 0} \frac{\log(1-z)}z = -1$, the function $\frac{\log(1-z)}z$ is bounded on every interval of the form $(0,a]\subset (0,1)$.  Applying this for $a=z_0$, we obtain
$$\log A^{(n)}(z_0) = \sum_{i=1}^n -a_i\log(1-z_0^i) < C \sum_{i=1}^n a_i z_0^i,$$
where $C$ depends on $z_0$ but not on $n$.  As $a_i$ is subexponential, the right hand side in this inequality is bounded independent of $n$, so the sequence $A^{(n)}(z_0)$ is bounded, so it converges.
\end{proof}

In particular, when $a_i = 1$ for all $i$, we obtain the well-known fact that the partition function $p(n)$ is subexponential.

\begin{lem}
\label{Products of subexponentials}
Let $(a_i)_{i=1,2,\ldots}$ and $(b_i)_{i=1,2,\ldots}$ be two sequences of positive integers such that $a_i$ is subexponential and $b_i$ is arbitrary.
Define $c_k$ to be the maximum of $\prod_j a_j^{e_j}$ as $1^{e_1}2^{e_2}\cdots$ ranges over all partitions of $k$ with $e_i\le b_i$ for all $i$.
Then $c_k$ is subexponential.
\end{lem}

\begin{proof}
For all $\epsilon > 0$, there exists $r$ such that $a_i < (1+\epsilon)^i$ for $i>r$.  Thus
$$c_k < a_1^{b_1}\cdots a_r^{b_r} (1+\epsilon)^k.$$
\end{proof}

\section{Counting polynomials}
In this section, we introduce several sets of polynomials which are candidates for the set $\P$ in Proposition~\ref{Master}.

For $P(x)\in \F_{q^2}[x]$ a monic polynomial with non-zero constant term, we define
$$P^*(x) := \bar P(0)^{-1} x^{\deg P}  \bar P(1/x),$$
where $\bar P$ is the polynomial obtained from $P$ by applying the $q$-Frobenius automorphism to each coefficient.
In particular, if $P(x)\in \F_q[x]$, then $P^*(x) = P(0)^{-1} x^{\deg P} P(1/x)$.
Note that $P^*(x)$ is a monic polynomial of the same degree as $P$.  If $P(x) = \prod_i (x-r_i)$, then 
$$P^*(x) = \prod_i (x-1/\bar r_i) = \prod_i (1-r_i^{-q}).$$
Therefore, if $P=P^*$, then the roots of $P$, taken with multiplicity, form a union of orbits in $\bar\F_q^\times$ under the map
$x\mapsto x^{-q}$.   Any orbit under this map is stable by the $q^2$-Frobenius, so if
$P(x)$ is irreducible in $\F_{q^2}[x]$, its roots form a single orbit under $x\mapsto x^{-q}$.  If $P(x)\in \F_q[x]$ and $P$ is irreducible as a polynomial over $\F_q$, then it may form one orbit or two mutually reciprocal orbits.

Following \cite{LS}, we denote by $\L_n(q)$  the set of monic polynomials $P(x)\in\F_q[x]$ 
of degree $n$ such that $P(0)= (-1)^n$.  We define by $\U_n(q)$ the set of monic polynomials $P(x)\in \F_{q^2}[x]$ of degree $n$ such that
$P = P^*$ and $P(0)= (-1)^n$.  When $n$ is even, we denote by $\O_{n}(q)$  the set of monic polynomials $P(x)\in \F_q[x]$ of degree $n$ such that
$P=P^*$ and $P(0) = 1$.
For $c\in \F_q^\times$, we denote by  $\L_{n,c}(q)$  the set of monic polynomials $P(x)\in\F_q[x]$ 
of degree $n$ such that $P(0)= c$.
Likewise, for $c\in \F_{q^2}^\times$ with $c\bar c=1$,
let $\U_{n,c}(q)$ denote the set of monic polynomials $P(x)\in \F_{q^2}[x]$ 
of degree $n$ such that $P=P^*$ and $P(0)=c$.

\begin{lem}
\label{count}
Let $r$ be a positive integer and $q$ a prime power.  Then
$$|\L_{r+1}(q)| = |\U_{r+1}(q)| = |\O_{2r}(q)| = q^r.$$
\end{lem}

\begin{proof}
In each case, the leading coefficient and the constant coefficient are fixed.
For $\L_{r+1}(q)$, the remaining $r$ coefficients can be chosen independently from $\F_q$.
For $\U_{r+1}(q)$, if $0<i<(r+1)/2$, the $x^i$ coefficient can be any element of $\F_{q^2}$ and it uniquely
determines the $x^{r+1-i}$ coefficient.  That finishes the $\U$ case if $r+1$ is odd.  If it is even, the $x^{\frac{r+1}2}$ coefficient can be any element of $\F_q$,
so again $|\U_{r+1}(q)| = q^r$.  For $\O_{2r}(q)$, the $x^i$ coefficients can be chosen independently from $\F_q$ for $0<i\le r$, and the $x^i$ coefficient 
determines the $x^{2r-i}$ coefficient.
\end{proof}

\begin{prop}
\label{polynomial estimate}
There exists a positive real sequence $(\epsilon_i)_{i=1,2,\ldots}$ tending to $0$ such that
all of the following statements hold for all integers $r\ge 1$.
\begin{enumerate}
\item
Let $n=r+1$.  If $m>\sqrt{n}$ and $c\in \F_q^\times$, then the number of elements in $\L_{n,c}(q)$ with an irreducible factor whose degree is divisible by $m$ is less than $\epsilon_r |\L_{n,c}(q)|$.
Likewise, the number of elements with an irreducible factor whose degree is $>\sqrt{n}$ and divides $\ell-1$
for some prime divisor $\ell$ of $n$ is less than $\epsilon_r |\L_{n,c}(q)|$.
\item
Let $n=r+1$. If $m>\sqrt{n}$ and $c\in \F_{q^2}^\times$ with $c\bar c=1$, then the number of elements in $\U_{n,c}(q)$ with an irreducible factor whose degree is divisible by $m$ is less than $\epsilon_r |\U_{n,c}(q)|$.
Likewise, the number of elements with an irreducible factor whose degree is $>\sqrt{n}$ and divides $\ell-1$
for some prime divisor $\ell$ of $n$ is less than $\epsilon_r |\U_{n,c}(q)|$.
\item 
Let $n=2r$.  If $m>\sqrt{n}$, then the number of elements in $\O_{n}(q)$ with an irreducible factor whose degree is divisible by $m$ is less than $\epsilon_r |\O_{n}(q)|$.
\end{enumerate}
\end{prop}

\begin{proof}
A monic irreducible polynomial over $\F_q$ of degree $k$ corresponds to a $q$-Frobenius orbit of length $k$ in $\bar\F_{q}^\times$.
Any such orbits is contained in the $(q^k-1)$-roots of $1$ in $\bar\F_q$, 
so there are less than $q^k/k$ such polynomials.
Therefore, the number of monic polynomials of degree $n$ with constant term $c$ and an irreducible factor of degree $k$ is less than $\frac{q^{n-1}}{k}$.
Summing over multiples of $m$, the number of monic polynomials of degree $n$ with constant term $c$ and an irreducible factor whose degree lies in $m\Z$ is less than
$$\sum_{1\le i\le n/m}\frac{q^{n-1}}{mi} < \frac{q^{r} (1+\log n)}{m} = \frac{|\L_{n,c}(q)|(1+\log n)}{m},$$
which gives the first claim in part (1).  For the second claim,  
we note that $n$ has at most one prime divisor $\ell > \sqrt n$.
So it suffices to prove that the sum of $1/m$ over divisors $m$ of $n$ which are larger than $\sqrt n$ is $o((1+\log n )^{-1})$. 
This follows from the fact that the total number of divisors of any integer $n$ is $n^{o(1)}$.

For (2), we proceed in the same way, using the fact that $|\U_{n,c}(q)| = q^{n-1}$.  If $Q\in \U_{n,c}(q)$ and $P$ divides $Q$, then $P^*$ divides $Q$.  It follows that if $P\neq P^*$, then any
element of $\U_{n,c}(q)$ divisible by $P$ is the product of $P P^*$ and a polynomial in $\U_{n-2k,cP(0)^{q-1}}$.  If $P=P^*$, then any element of $\U_{n,c}(q)$ divisible by $P$ is the product of $P$
and an element of $\U_{n-k,cP(0)^{-1}}$.  The first case gives less than
$$\frac{q^{2k}}k q^{n-2k-1} = \frac{q^{r}}k$$
elements of $\U_{n,c}(q)$.

For the second term, every monic irreducible degree $k$ polynomial $P(x)\in \F_{q^2}[x]$  such that $P=P^*$ corresponds
to a length-$k$ orbit 
$$r,r^{-q},r^{q^2},\ldots,r^{(-q)^k}=r,$$
so $r$ is a $(q^k-(-1)^k)$-root of $1$.  If $k\ge 2$, the $q+1$ fixed points of $x\mapsto x^{-q}$ do not belong to such an orbit, so the number of orbits
is again less than $q^k/k$, thus contributing less than
$$\frac{q^{k}}k q^{n-k-1} = \frac{q^{r}}k$$
elements of $\U_{n,c}(q)$.
The argument therefore goes through as before.

For (3), we follow (2).  The number of elements in $\O_{n}(q)$ is $q^r$.  
If a monic polynomial $P(x)\in \F_q[x]$ satisfies $P\neq P^*$ and $P$ divides $Q\in \O_{n}(q)$, then $Q$ is the product of $PP^*$ and an element
of $\O_{n-2k}(q)$.  For $k\ge 2$, a monic irreducible polynomial $P$ of degree $k$ satisfying $P=P^*$ must be of even degree, and every element of $\O_{n}(q)$
divisible by $P$ is the multiple of $P$ by an element of $\O_{n-k}(q)$.  By the proof of \cite[Prop.~2.6]{LS}, the number of irreducible monic polynomials of degree $k\ge 4$
satisfying $P=P^*$ is the same as the number of monic irreducible polynomials of degree $k/2$, and for $k=2$, the number is at most $q-1$.  Thus, the argument goes through as before.

\end{proof}

For any monic polynomial $P(x)$ over a field $F$, we define $\rho(P)$ to be the sum $\sum_{i=1}^j b_i (a_i-1)$, where $P = P_1^{a_1}\cdots P_j^{a_j}$, and the $P_i$ are pairwise distinct monic irreducible polynomials over $F$ of degree $b_i$.  For a perfect field $F$, $\rho(P)$ does not change if $F$ is replaced by a field extension.

\begin{lem}
\label{Exponential decay}
Let $m$ and $n$ be positive integers.
\begin{enumerate}
\item The number of polynomials $P\in \L_n(q)$ with $\rho(P)\ge m$ is less than $2q^{-m/2}|\L_n(q)|$.
\item The number of polynomials $P\in \U_n(q)$ with $\rho(P)\ge m$ is less than $4q^{-m/2}|\U_n(q)|$.
\item If $n$ is even, the number of polynomials $P\in \O_{n}(q)$ with $\rho(P)\ge m$ is less than $2q^{-m/4}|\O_{n}(q)|$.
\end{enumerate}
\end{lem}

\begin{proof}
Let 
$$Q = \prod_{i=1}^j P_i^{\lfloor a_i/2\rfloor},\ R = \frac{P}{Q^2}.$$
For claim (2) (resp.\ (3)), if $P_i^* = P_j$, then $a_i = a_j$, so the multiplicities of $P_i$ and $P_j$ in $Q$ (or in $R$) are the same.  Therefore, $Q\in \U_{\deg Q,Q(0)}(q)$ and $R\in \U_{\deg R,P(0)Q(0)^{-2}}(q)$
(resp.\ $Q\in \O_{\deg Q}(q)$ and $R\in \O_{\deg R}(q)$.)  
As 
$$\sum_i  b_i\lfloor a_i/2\rfloor \ge \frac 12\sum_i b_i(a_i-1),$$
we have $\deg Q \ge \rho(P)/2$ and $2\deg Q + \deg R = n$.

For $\L_n(q)$, the total number of possibilities for $(Q,R)$ with $\rho(P)\geq m$ is therefore at most
$$\sum_{i\ge m/2}q^{i}q^{n-2i-1} < 2q^{n-1-m/2}.$$
For $\U_n(q)$, there are $q+1$ elements $c\in\F_{q^2}$ with $c\overline{c}=1$, and for each of these $|\U_{k,c}|=q^{k-1}$, so the total
number of possibilities for $(Q,R)$ with $\rho(P)\geq m$ is at most
$$\sum_{i\ge m/2}(q+1)q^{i-1}q^{n-2i-1} < 4q^{n-1-m/2}.$$
For $\O_n(q)$, the number of possibilities is at most 
$$\sum_{i\ge m/2}q^{i/2}q^{(n-2i)/2} < 2q^{n/2-m/4}.$$
\end{proof}

With $a_i$, $b_i$, and $j$ as above, we define 
$$\alpha(P) := \frac{\prod_{i=1}^j (1+q^{-b_i})}{1-q^{-1}}.$$

\begin{lem}
\label{Undersized classes are rare}

For all $\epsilon > 0$ there exists $M$ such that for all $n\ge 1$ and all prime powers $q$:
\begin{enumerate}
\item The number of polynomials $P\in \L_n(q)$ with $\alpha(P)<(1+q^{-1})^M$ is greater than $(1-\epsilon)|\L_n(q)|$.
\item The number of polynomials $P\in \U_n(q)$ with $\alpha(P)<(1+q^{-1})^M$ is greater than $(1-\epsilon)|\U_n(q)|$.
\item If $n$ is even, the number of polynomials $P\in \O_{n}(q)$ with $\alpha(P)<(1+q^{-1})^M$ is greater than $(1-\epsilon)|\O_{n}(q)|$.
\end{enumerate}
\end{lem}

\begin{proof}
As there exists $C$ such that $(1-q^{-1})^{-1}\prod_{i\geq 1} (1+q^{-i})<(1+q^{-1})^C$ for all $q\geq 2$, 
it suffices to prove that there exists $e\ge 2$ such that the fraction of polynomials $P$ in $\L_n(q)$ (resp.\ $\U_n(q)$ or $\O_{n}(q)$)
divisible by $e$ different irreducible polynomials of the same degree is less than $\epsilon$.  The number of monic irreducible
polynomials of degree $k$ in $\F_q[x]$ is less than $q^k/k$, so the number of sets $\{Q_1,\ldots,Q_k\}$ of $e$ such polynomials is at most $k^{-e} q^{ek}/e!$.
For each possibility, there are $q^{n-ek-1}$ choices for the remaining factor $P/\prod_i Q_i$, so there are at most $(k^{-e}/e!) q^{n-1}$ elements of $\L_n(q)$ with
$e$ distinct irreducible degree $k$ factors.  Summing over $k$, we get an upper bound of $\zeta(2)q^{n-1}/e!$, and $\zeta(2)/e!$ goes to $0$ as
$e$ goes to $\infty$.

For statements (2) and (3) we consider factors of degree $k$ which are either of the form $Q$, where $Q=Q^*$ is irreducible (in $\F_{q^2}[x]$ or $\F_q[x]$ respectively)
or which are of the form $Q Q^*$, where $Q$ is of degree $k/2$ and $Q\neq Q^*$.
As in Lemma~\ref{polynomial estimate}, the two cases together contribute less than $2q^k/k$ possibilities.
The number of $e$-element sets of such polynomials is therefore less than $(2k)^{-e} q^{ek}/e!$, and the argument goes through as before.
\end{proof}

\begin{lem}
\label{A orbits}
Let $n\ge 2$ be an integer and $q$ a prime power.  The number of elements $P\in \L_n(q)$ such that
$$P(\zeta x)\equiv P(x)$$
for some $\zeta \in \F_q \setminus \{1\}$ is less than $2q^{n/2-1}$.
Likewise, the number of elements $P\in \U_n(q)$ such that
$$P(\zeta x)\equiv P(x)$$
for some $\zeta\in \F_{q^2}\setminus\{1\}$ is less than $2q^{n/2-1}$.
\end{lem}

\begin{proof}
In both cases, if $P(\zeta x)\equiv P(x)$ for some $\zeta\neq 1$,
then by comparing coefficients, the order $d>1$ of $\zeta$ divides
 $n$, and $P(x)=P'(x^d)$ for some element
$P'\in\L_{n/d,(-1)^n}(q)$ (resp.\ $\U_{n/d,(-1)^n}(q)$), for which there are
$q^{n/d-1}$ possibilities.  So the number of $P\in\L_n(q)$ (resp.\ $\U_n(q)$)
with $P(\zeta x)\equiv P(x)$ for some $\zeta\in\F_q\setminus \{1\}$
(resp.\ $\F_{q^2}\setminus \{1\}$) is at most
$$\sum_{2\le i\le n/2}q^{n/i-1} < 2q^{n/2-1}.$$

\end{proof}

\section{Classical groups}

Finite simple groups $G$ of rank $r>8$ must be of type $A_r$, $^2A_r$, $B_r$, $C_r$, $D_r$, or $^2D_r$.  In each case, $G$ is closely related to a \emph{classical group} $G'$, which we define below.  We also define the sets $\A$, $\At$, $\B$, $\Bt$, and $\P$ used in Proposition~\ref{Master}.

In every case, $\A$ denotes the set of conjugacy classes of $G$ and
$\At$, the set of conjugacy classes of the universal central extension $\tilde G$ of $G$.  
We denote by $Z$ the center of $\tilde G$, so that we can think of $|Z|$ as the ``generic''
size of the fibers of
the map $\phi$ obtained from the covering homomorphism $\tilde G\to G$ by taking conjugacy classes.
The map $f$ takes a conjugacy class of $\tilde G$ to its characteristic polynomial, with a slight modification when $G$ is of type B.

We can regard $G$ as the commutator group $[\uG_{\ad}(\F_q),\uG_{\ad}(\F_q)]$, where
$\uG_{\ad}$ is an adjoint simple algebraic group defined over $\F_q$.  We define $\B := \Irr(G)$, while $\Bt$  denotes the set of pairs $(\chi,\chi_{\ad}) \in \Irr(G)\times \Irr(\uG_{\ad}(\F_q))$
such that $\langle \chi,\Res_G \chi_{\ad}\rangle_G \ge 1$.  We define $\psi(\chi,\chi_{\ad}) := \chi$.  The Lusztig classification (see \S8 below) assigns to each character $\chi_{\ad}$
a semisimple conjugacy class in the $\F_q$-points of the dual group to $\uG_{\ad}$.  This is a simply connected simple algebraic group of classical type, so it has a natural representation,
and we define $g((\chi,\chi_{\ad}))$ to be the characteristic polynomial of this semisimple class in this natural representation, with a slight modification in the case that $G$ is of type C.

We divide into cases.  A reference for dual groups for the various groups of classical types is \cite[\S2]{Carter}.

\medskip\noindent \textbf{Case A.}  In this case, $G$ must be of the form $\PSL_n(\F_q)$
or $\PSU_n(\F_q)$, where $n=r+1$.  We define $G'$ to be  $\SL_n(\F_q)$ or $\SU_n(\F_q)$
respectively.  As $\uG_{\ad}$ is $\PGL_n$ or $\PGU_n$ respectively, the dual group $(\uG_{\ad})^*$ is $\SL_n$ or $\SU_n$ respectively,
and $\P=\L_n(q)$.  We have $|Z|\le n$.

\medskip\noindent \textbf{Case B.} 
In this case, $G$ is of the form $\Omega_n(\F_q)$, where $n=2r+1$, and we define $G'$ to be $\SO_n(\F_q)$.  In this case, $G'$ is a subgroup of $G$ of index $\le 2$. 
As $\uG_{\ad} = \SO_n$, the dual group $(\uG_{\ad})^*$ is $\Sp_{2r}$.  The characteristic polynomial of every element of $\Sp_{2r}(\F_q)$ lies in $\P := \O_{2r}(q)$.
The characteristic polynomial of every element of $\SO_n(\F_q)$ is $(x-1)$ times an element of $\O_{2r}(q)$, and we define $f$ as the composition of  $\tilde G\to G$,
$G\to \GL_n(\F_q)$, the characteristic polynomial map, and division by $(x-1)$. 
We have $|Z|\le 2$.

\medskip\noindent \textbf{Case C.} 
In this case, $G$ is of the form $\PSp_n(\F_q)$, where $n=2r$, and we define $G' $ to be $\Sp_n(\F_q)$, so $G$ is the quotient of $G$ by a normal subgroup of order $\le 2$,
and $f$ is defined via the usual map $\Sp_n(\F_q)\to \O_n(\F_q)$.
As $\uG_{\ad} = \PGSp_n$, the dual group $(\uG_{\ad})^*$ is $\Spin_{2r+1}$.  We define the map $g$ by composing the maps $\Spin_{2r+1}(\F_q)\to \SO_{2r+1}(\F_q)$,
$\SO_{2r+1}(\F_q)\to \GL_{2r+1}(\F_q)$, the characteristic polynomial map, and division by $(x-1)$.
We have $|Z|\le 2$.

\medskip\noindent \textbf{Case D.} 
In this case, $G$ is of the form $P\Omega^{\pm}_n(\F_q)$, where $n=2r$, and we define $G' $ to be $\SO^{\pm}_n(\F_q)$,  so $G$ is the quotient of a subgroup $\Omega^{\pm}_n(\F_q)$ of index $\le 2$ in $G'$ by a normal subgroup of order $\le 2$.  As $\uG_{\ad} = \PO^\pm_n$, the dual group $(\uG_{\ad})^*$ is $\Spin^\pm_n$.  Both $f$ and $g$ are defined by composing
$\Spin^\pm_n(\F_q)\to \SO^\pm_n(\F_q)$, $\SO^\pm_n(\F_q)\to \GL_n(\F_q)$, and the characteristic polynomial map, which sends orthogonal $n\times n$ matrices to elements of
$\O_{n}(q)$.  Note that in this case $f$ and $g$ are not surjective.  We have $|Z|\le 4$.

\begin{lem}
\label{Conditions 7 and 8}
In all four cases, conditions (7) and (8) of Proposition~\ref{Master} hold if $N>2$.
\end{lem}

\begin{proof}
By \cite[Theorem~1.1]{FG}, $k(\tilde G) \ge q^r$.  By Lemma~\ref{count}, 
$$|\At| = k(\tilde G) \ge q^r = |\P|$$
This implies condition (7).   As the projection map from $\Bt$ to $\Irr(\uG_{\ad}(\F_q))$ is surjective, 
$$|\Bt| \ge \Irr(\uG_{\ad}(\F_q)) = k(\uG_{\ad}(\F_q)) \ge q^r = |\P|,$$
again by \cite[Theorem~1.1]{FG}.

\end{proof}

\begin{prop}
Let $G'\subset \GL_n(\F_{q^2})$ be a classical group.  The characteristic polynomial of every element of $g$ belongs to $\L_n(q)$, $\U_n(q)$,  $(x-1)\O_{n-1}(q)$, $\O_n(q)$, $\O_n(q)$, or $\O_n(q)$
if $G'$ is of type $A$, $^2A$, $B$, $C$, $D$, or $^2D$ respectively.  Moreover, for each such element, there exist at most $4$ semisimple conjugacy classes in $G'$ whose elements
have this characteristic polynomial.
\end{prop}

\begin{proof}
The first part is well-known; see, e.g., \cite{LS}.  The second part follows from the following two claims.  First, we assert the map from the variety of semisimple conjugacy classes of the underlying linear, unitary, symplectic, or orthogonal algebraic group to the variety of conjugacy classes of $\GL_n$ is at most $2$ to $1$.  Second, we assert that the elements of $G'$ in any semisimple conjugacy class of the underlying algebraic group $\uG$ split into at most $2$ $G'$-conjugacy classes.  

For the first assertion, we may work over $\F_q$ and fix a maximal torus $\uT$ of $\uG$ which lies in the maximal torus $D$ of diagonal elements in $\GL_n$.  Let $W$ denote the Weyl group 
of $\uG$ with respect to $\uT$ and consider the map $T/W\to D/S_n$.  We claim that for any $t\in \uT$, there are at most $2$ different $W$-orbits in $\uT\cap t^{S_n}$.  This is obvious for type A.  If $n=2r$, two $n$-tuples of the form
$$(x_1,x_1^{-1},\ldots,x_r,x_r^{-1})$$
are the same up to rearrangement if and only if the multisets 
$$\{\{x_1,x_1^{-1}\},\ldots,\{x_r,x_r^{-1}\}\}$$
are the same, and this implies that the $n$-tuples lie in the same $(\Z/2\Z)^r\rtimes S_r$-orbit.
This shows that the map $T/W\to D/S_n$ is one-to-one in case C and at most $2$ to $1$ in case D.
If $n=2r+1$, then two $n$-tuples
$$(x_1,x_1^{-1},\ldots,x_r,x_r^{-1},1)$$
are the same up to rearrangement if and only if the $n$-tuples lie in the same $(\Z/2\Z)^r\rtimes S_r$-orbit, so again the map is one-to-one.

For the second assertion, we use the fact that the map from the universal cover of $\uG$ to $\uG$ is at most $2$ to $1$.  From Steinberg's theorem \cite[Theorem~9.1]{Steinberg} it follows 
if $\uZ_s$ is the centralizer of a semisimple element in $\uG$, then $\uZ_s/\uZ_s^\circ$ is 
of order $1$ or $2$.
By Lang's theorem, it follows that there are at most two $G'$-conjugacy classes of elements in $G'$ conjugate to $s$ under $\uG$.

\end{proof}

\section{Unipotent conjugacy classes}

\begin{lem}
\label{Simple Growth}
The sequence whose $r$th term is maximum number of unipotent conjugacy classes in any 
classical group of rank $r$ over any field $\F_q$ has subexponential growth.
\end{lem}

\begin{proof}
By \cite[Prop.~2.1]{GLO}, the number of unipotent conjugacy classes in $\SL_n(\F_q)$ is $\le n p(n)$.
By \cite[Prop.~2.2]{GLO}, the same bound applies for $\SU_n(\F_q)$.
By \cite[Prop.~2.3]{GLO}, for a symplectic group of rank $r$, the number of unipotent conjugacy classes is
the sum of $2^{a_\lambda}$ over partitions $\lambda$ of $2r$, where $a_\lambda$ denotes the number of distinct even parts.
Since the sum of $a_\lambda$ distinct positive even integers is at least $a_\lambda^2+a_\lambda \le 2r$, it follows that the maximum of $2^{a_\lambda}$ is
subexponential in $r$, as is $p(2r)$.
By  \cite[Prop.~2.4]{GLO}, for any orthogonal group of rank $r$, the number of unipotent conjugacy classes is
the sum of $2^{a_\lambda}$ over partitions $\lambda$ of $2r$, where $a_\lambda$ is one less than the number of odd parts in $\lambda$,
with the exception that if $\lambda$ has no odd parts, the summand is either $0$ or $1$, depending on whether $G$ is of the form $\SO^-$ or $\SO^+$.

For $G$ either orthogonal or symplectic and $q$ even, \cite[Prop.~3.1]{GLO} gives a more complicated classification of unipotent conjugacy classes, 
but the number of representations is certainly bounded above by ordered quadruples of partitions summing to $r$, which is the $z^r$ coefficient of
$\prod_{i=1}^\infty (1-z^i)^{-4}$ and therefore subexponential in $r$.

\end{proof}

\begin{prop}
\label{Regularity}
For all $\epsilon > 0$ there exists $N$ with the following property.
For any finite field $\F_q$, any $n>N$, and any semisimple element $s$ in a classical subgroup $G'=\uG'(\F_q)$ of $\GL_n(\F_q)$,
let $\uH$ be the centralizer of $s$ in $\uG'$, 
 $\uH^\circ$  the identity component of $\uH$,  $\uS$ the derived group of $\uH^\circ$, and $r$ the absolute rank of $\uS$.
Then the number of $\uH(\F_q)$-conjugacy classes of unipotent elements in $\uH^\circ(\F_q)$ is less than $q^{\epsilon r}$.
The analogous statement is also true when $\uH$ is the centralizer of a semisimple element $s$ in $G' = \SU_n(\F_q)$.
\end{prop}

\begin{proof}
It suffices to prove that the number of conjugacy classes of unipotent elements in $\uH^\circ(\F_q)$ is subexponential.   As $\uH^\circ/\uS$ is diagonal, every unipotent element of $\uH^\circ(\F_q)$
lies in $\uS(\F_q)$.  Thus it suffices to prove a subexponential bound for the unipotent conjugacy classes of $\uS(\F_q)$.

We decompose the natural representation space of $\F_q^n$ of $\GL_n(\F_q)$ by $s$ into $s$-isotypic factors $V_Q  \cong W_Q^{a_Q}$  indexed by  monic irreducible polynomials $Q(x)\in \F_q[x]$ and denote by $b_Q$ the dimension $\dim W_Q = \deg Q$.  If $G'=\SL_n(\F_q)$, then 
$$\uS = \prod_Q \Res_{\F_{q^{b_Q}}/\F_q} \SL_{a_Q,\F_q},$$
where $\Res$ denotes restriction of scalars.  
Each factor is of rank $b_Q(a_Q-1)$ over $\F_q$, so the rank of $\uS$ is $\rho(P)$, where $P$ is the characteristic polynomial of $s$.

For orthogonal groups $G'$,  let $\Pi$ denote the set of orbits for the involution $Q\mapsto Q^*$.  For $\pi\in\Pi$, we denote by $V_\pi$, $W_\pi$, $a_\pi$, and $b_\pi$ 
the sum $\bigoplus_{Q\in\pi} V_Q$, $\bigoplus_{Q\in\pi}W_Q$, $a_Q = a_{Q^*}$, and $\dim W_\pi$ respectively.
As $s$ preserves the inner product $\langle\,,\,\rangle$, we have $V_Q\perp V_R$ unless $Q=R$, so the centralizer of $s$ in $G'$ is 
\begin{equation}
\label{decomp}
\prod_{\pi\in\Pi} \Aut_{\F_{q^{b_\pi}}}(V_\pi,\langle\,,\,\rangle).
\end{equation}
The derived group of the identity component is therefore a product of simple algebraic groups $\uS_\pi$ indexed by $\pi\in \Pi$.
If $\pi = \{x-1\}$ or $\pi= \{x+1\}$,  then $\uS_\pi$ is of absolute rank $\lfloor a_x/2\rfloor$ and of type D or B as $a_x$ is even or odd.  
Otherwise, it is of type D or A,
depending on whether $\pi$ has one element or two
and of absolute rank $b_\pi (a_\pi - 1)$ in either case.
For symplectic groups, we proceed in the same way, with the difference that polynomials $Q=Q^*$ give rise to factors of type C.

For unitary groups, $G'$ acts on an $n$-dimensional vector space over $\F_{q^2}$.  We decompose $\F_{q^2}^n$ into isotypical spaces $V_Q\cong W_Q^{a_Q}$ for the action of $s$, where $Q$ ranges over monic irreducible polynomials in $\F_{q^2}[x]$.
Let $\Pi$ denote the set of orbits of $\{Q\mid a_Q>0\}$ under $Q\mapsto Q^*$.  Let $W_\pi = \bigoplus_{Q\in \pi}W_Q$ and
$V_\pi = \bigoplus_{Q\in \pi}V_Q = W_\pi^{a_\pi}$.
Let $\langle\,,\,\rangle$ denote the sesquilinear blinear form which $G'$ respects.  The different $V_\pi$ are mutually
orthogonal with respect to this form, 
and the centralizer of $s$ in $G'$ is again given by \eqref{decomp}.
The derived group is therefore a product of simple algebraic groups indexed by $Q$, and each is of type A and absolute rank $b_\pi(a_\pi-1)$,
where $b_\pi = \dim W_\pi$.

In every case, therefore, $\uS(\F_q)$ is a product of classical groups of total rank $r\le\rho(P)$.  
The number of unipotent conjugacy classes is therefore $\le \prod_Q c_{b_Q}$, where $(c_i)_{i=1,2,\ldots}$ is the subexponential sequence given by Lemma~\ref{Simple Growth}.
By Lemma~\ref{Products of subexponentials}, for any fixed $q$, the number of conjugacy classes is $O(q^{\epsilon r/2})$
and therefore less than $q^{\epsilon r}$ if $r$ is sufficiently large.  On the other hand, there exists $\alpha$ such that $c_i < \alpha^i$ for all $i$.  If $q > \alpha^{1/\epsilon}$, then
the number of conjugacy classes is less than or equal to $\alpha^r \le q^{\epsilon r}$.

\end{proof}

\section{Unipotent characters}

Let $\uG$ be a connected reductive group over $\F_q$.  Following Lusztig \cite{Lusztig}, we say that an irreducible character of $\uG(\F_q)$ is \emph{unipotent}
if it appears with non-zero multiplicity in the Deligne-Lusztig representation $R_{\uT}^{\uG}(1)$ associated to the trivial character on maximal torus $\uT(\F_q)$.  In particular,
the trivial character is unipotent.  The classication of unipotent characters depends only on the adjoint quotient of $\uG$
(see \cite[Remark]{Lusztig-Disconnected}), therefore only on the root system of $\uG$ together with Frobenius action.

Assuming $\uG$ has connected center, Lusztig gave \cite[p.~x]{Lusztig} a ``Jordan decomposition" of irreducible characters $\chi$ of $\uG(\F_q)$.  We briefly recall the setup, referring the reader to \cite{Lusztig} for details.  Each such character has non-zero multiplicity
in some Deligne-Lusztig character $R_{\uG}^{\uT}(\theta)$, and $\theta$ determines a semisimple element $t$ of the dual group $\uG^*(\F_q)$, where $\uG^*$
is the connected reductive algebraic group over $\F_q$ whose root datum is dual to that of $\uG$, with corresponding Frobenius action.  The element $t$ is well-defined by
up to 
conjugacy class by $\chi$.  As $\uG$ has connected center, the derived group of $\uG^*$ is simply connected, so choosing a representative $t$, the centralizer $\uH^*$ of $t$ in $\uG^*$
is a connected reductive group.  If $\uH$ denotes the dual group of $\uH^*$, there is a bijective correspondence $\pi\mapsto \chi_\pi$ between the set of unipotent characters $\pi$ of $\uH(\F_q)$
and the set $\E(t)$ of irreducible characters $\chi_\pi$ of $\uG(\F_q)$ associated to the class of $t$.
For us, the most important point is that 
\begin{equation}
\label{Induction degree}
\chi_\pi(1) = \frac{|\uG(\F_q)|'}{|\uH(\F_q)|'}\pi(1),
\end{equation}
where $m'$ denotes the largest divisor of $m$ prime to $q$.

We record the following consequence.

\begin{lem}
\label{ord-l of degree}
If $\chi_{\ad}$ is a character of $\uG_{\ad}(\F_q)$ associated to the class of a semisimple element
$t\in (\uG_{\ad})^*(\F_q)$, and if the order of the centralizer of $t$ is not divisible by a prime $\ell\nmid q$, then
$$\ord_\ell \chi_{\ad}(1) = \ord_\ell |\uG_{\ad}(\F_q)|.$$

\end{lem}

\begin{proof}
As $\ell\nmid q$, we have $\ord_\ell |\uG_{\ad}(\F_q)| = \ord_\ell |\uG_{\ad}(\F_q)|'$.
Defining $\pi$ so that $\chi_{\ad} = \chi_\pi$, by
 \eqref{Induction degree},
$$\ord_\ell \chi_{\ad}(1) = \ord_\ell |\uG_{\ad}(\F_q)| \pi(1)
 \ge \ord_\ell |\uG_{\ad}(\F_q)|.$$
The opposite inequality follows from the fact that $\chi_{\ad}(1)$ divides $|\uG_{\ad}(\F_q)|$.
\end{proof}

\begin{prop}
\label{Unipotent characters}
For all $\epsilon > 0$ there exists $N$ with the following property.
For any finite field $\F_q$, any $r>N$, any adjoint simple group $\uG$ over $\F_q$
of type A, B, C, or D, and any semisimple element $t\in \uG^*(\F_q)$,
such that the centralizer of $t$ in $\uG^*$  has absolute semisimple rank $r$,
the number of elements in $\E(t)$ is less than $q^{\epsilon r}$
\end{prop}

\begin{proof}
The proof is essentially the same as that of Proposition~\ref{Regularity}.  The only difference
is that instead of Lemma~\ref{Simple Growth}, we use a subexponential estimate for the number
of unipotent characters of a classical simple group of rank $r$.
The number of unipotent characters is independent of $q$.
For special linear and unitary groups, it is given by the partition function $p(r)$
\cite[p.~358]{Lusztig}.  For orthogonal and symplectic groups, there are at most two different unipotent characters associated to a Lusztig symbol of rank $r$ \cite[p.~359]{Lusztig}.  The number of such symbols grows subexponentially by \cite[Prop.~3.4]{Lusztig-Classical} and Proposition~\ref{subexp}.
\end{proof}

The irreducible characters of the finite simple group $G$ can be regarded as the $Z$-trivial characters of $\tilde G$, where $Z$ is the center of $\tilde G = \uG_{\sc}(\F_q)$.
By \cite[Prop.~5.1]{Lusztig-Disconnected}, $\Irr(\tilde G)$ can be decomposed into classes $\E(s)$ indexed by semisimple conjugacy classes in $(\uG_{\sc})^*(\F_q)$.
Moreover, the conjugation action of $\uG_{\ad}(\F_q)/G$ on $\Irr(G)$ preserves this decomposition, and the orbits corresponding to elements of $t\in (\uG_{\sc})^*(\F_q)$ with connected centralizer
are singletons.  For such $s$, therefore, each character of $G$ extends to $|Z|$ different characters of $\uG_{\ad}(\F_q)$, obtained from one another by tensor product by $1$-dimensional characters
of $\uG_{\ad}(\F_q)/G$ (which are necessarily trivial on $G$).  Thus the correspondence between $\Irr(\uG_{\ad}(\F_q))$ and $\Irr(G)$ is given by a function (namely, restriction) on the complement of
the set of characters of $\Irr(G)$ corresponding to $t$ with disconnected centralizer.  If $\tilde t\in (\uG_{\ad})^*(\F_q)$ is a lift of $t$ to an element on the universal cover, then $t$ fails to have
connected centralizer only if the multiple of $\tilde t$ by some non-trivial central element is conjugate to $s$ and therefore only if the characteristic polynomial of $\tilde t$ is 
a polynomial $P(x)$ satisfying $P(\zeta x)\equiv P(x)$ for some $\zeta\neq 1$.

\section{Zsigmondy primes}

We recall that given $q$ and $m$, a \emph{Zsigmondy prime for the pair $(q,m)$} is a prime $\ell$ such that $q$ has order exactly $m$ in $\F_\ell^\times$.
Zsigmondy's theorem asserts that such a prime always exists if $m>6$.

\begin{lem}
\label{Which Z}
If $\ell$ is a Zsigmondy prime for $(q,m)$, then $\ell$ divides $q^k-1$ if and only if $m$ divides $k$, and $\ell$ divides $q^k+1$ if and only if  $k$ is an odd
multiple of $m/2$.
\end{lem}

\begin{proof}
The condition that $\ell$ divides $q^k-1$ is equivalent to the condition that the $k$th power of $q$ in $\F_\ell^\times$ is $1$, i.e., that $m$ divides $k$.
The condition that $\ell$ divides $q^k+1$ is equivalent to the condition that $m$ divides $2k$ but not $k$, i.e., $dm=2k$ for some odd integer $d$.  Equivalently
$k$ is an odd multiple of $m/2$.
\end{proof}

\begin{lem}
\label{Just one}
If a semisimple element $s\in \SL_n(\F_q)$ has characteristic polynomial $P$, $m>2\rho(P)$, and no irreducible factor of $P$ has degree a multiple of $m$, then any Zsigmondy prime $\ell$ for $(q,m)$
is relatively prime to the order of the centralizer of $s$ in $\SL_n(\F_q)$.  Likewise, if $m>4\rho(P)$, $s\in \SU_n(\F_q)$, and no irreducible $\F_{q^2}[x]$ factor of $P$ has degree an integer multiple of $m/2$,
then any Zsigmondy prime $\ell$ for $(q,m)$ is relatively prime to the order of the centralizer of $s$ in $\SU_n(\F_q)$.

\end{lem}

\begin{proof}
For $\SL_n(\F_q)$, it suffices to prove that $\ell$ does not divide the order of the centralizer of $s$ in $\GL_n(\F_q)$.
Factoring
$$P = \prod_{i=1}^j Q_i^{a_i},$$
with $\deg Q_i = b_i$,
the centralizer of $s$ can be written
$$\prod_{i=1}^j \GL_{a_i}(\F_{q^{b_i}}).$$
For each $i$, $b_i(a_i-1) \le \rho(P)$, so if $a_i\ge 2$, we have $a_i b_i \le 2\rho(P) < m$, so
$$|\GL_{a_i}(\F_{q^{b_i}})| = \prod_{k=0}^{a_i-1} q^{b_ik} (q^{b_i(a_i-k)}-1)$$
is prime to $\ell$.  If $a_i=1$, then $\GL_1(\F_{q^{b_i}})$ has order $q^{b_i}-1$ which is again prime to $\ell$.

For $\SU_n(\F_q)$, we proceed as before, computing the centralizer of $s$ in $U_n(\F_q)$ as  in Proposition~\ref{Regularity}.
In this case, the centralizer factors are of the form $U_{a_i}(\F_{q^{b_i}})$ or $\GL_{a_i}(\F_q^{2b_i})$ depending on whether $Q_i=Q_i^*$.  As $a_i\ge 2$ implies $2a_ib_i \le 4\rho(P)<m$, it follows that no Zsigmondy prime for $(q,m)$
can divide the order of a factor of either kind, so we may assume $a_i=1$.
As $b_i$ is not a multiple of $m/2$, $\ell$ divides neither the order of $\GL_1(\F_q^{2b_i})$ nor $U_1(\F_q^{b_i})$.

\end{proof}

We remark that if $\ell$ divides $n$,  then $m$ divides $\ell-1$ for a prime divisor $\ell$ of $n$.

\section{Weak regularity conditions}

Let $k\ge 1$ and $m\ge 0$ be integers.  We say a polynomial $P(x)\in \bar\F_q[x]$ is 
$m$-\emph{regular} if the following two conditions hold:
\begin{enumerate}
\item $\rho(P)\le m$.
\item $P(x)$ is not identical to $P(\zeta x)$ for any $\zeta\neq 1$.
\end{enumerate}

If the characteristic polynomial of an element in $\GL_n(\bar\F_q)$ is $m$-regular, 
we say that this element is $m$-regular.  This depends only on the semisimple part $s$ in the Jordan 
decomposition the element.  Likewise, we say an irreducible character of $\uG_{\ad}(\F_q)$ is $m$-regular if and only if 
it belongs to $\E(s)$, where the characteristic polynomial of the image of $s$ under the natural representation of $(\uG_{\ad})^*(\F_q)$
is $m$-regular.  

Let $G$ be a classical finite simple group.
We define $G'$, $\A$, $\At$, $\B$, $\Bt$, $\P$ as in \S5.  Given a fixed choice of $m$,
we define $\P^\circ$ to be the subset of $m$-regular polynomials in $\P$, $\At^\circ := f^{-1}(\P^\circ)$, $\A^\circ := \phi(\At^\circ)$, $\Bt^\circ := g^{-1}(\P^\circ)$, $\B^\circ := \psi(\Bt^\circ)$.
Note that $\phi^{-1}(\A^\circ) = \At^\circ$ since $\rho(P(x)) = \rho(\omega^{-\deg P}P(\omega x))$ for all scalars 
$\omega\neq 0$.  Likewise, $\psi^{-1}(\B^\circ) = \Bt^\circ$, since if $(\chi,\chi_{\ad})$ and
$(\chi,\chi'_{\ad})$ both lie in $\Bt$, and $s\in \tilde G$ lies in the semisimple class associated to $\chi_{\ad}$, then there exists $z\in Z$ such that $zs$ lies in the semisimple class associated to $\chi'_{\ad}$.

By definition, the image of any element of $\At^\circ$ in $G'$ is $m$-regular.  Part (2) of the definition of $m$-regularity guarantees that the fibers of $\phi$ over $\A^\circ$ and of $\psi$ over $\B^\circ$ all have exactly $|Z|$ elements, where $Z$ is the center of $\tilde G$.
Since $\tilde G\to G$ is $|Z|$ to $1$ and $G$ is of index $|Z|$ in $\uG_{\ad}(\F_q)$, all  fibers of $\phi$ and $\psi$ have cardinality $\le |Z|$.  This gives conditions (3) and (4) of Proposition~\ref{Master}.

If $s\in G'$ is semisimple and $m$-regular, its centralizer in $G'$ is the group of
$\F_q$-points of a reductive algebraic group $\uG$ over $\F_q$.
We have seen that $\uG$ has at most $2$ components, so if $q$ is odd, every unipotent element $u\in G'$ which commutes with $s$ lies in $\uG^\circ(\F_q)$.
If $q$ is even, we can regard $G'$ as the group of $\F_q$-points of a simply connected semisimple group, so the centralizer of $s$ is connected, and again $u\in \uG^\circ(\F_q)$.
To bound the number of $G'$-conjugacy classes of elements in $G'$ with semisimple part conjugate to $s$,
it suffices to bound the  $\uG^\circ(\F_q)$-conjugacy classes of unipotent elements in $\uG^\circ(\F_q)$.
By Proposition~\ref{Regularity}, we have a subexponential bound for this quantity.  As the homomorphism $\tilde G\to G'$ is at most $2$ to $1$, we have a subexponential bound in $m$  for 
the number of elements of $\At^\circ$ mapping to any element of $\P^\circ$, the set of $m$-regular polynomials in $\P$.
Likewise, by Proposition~\ref{Unipotent characters}, we have a subexponential bound for the number of elements of $\Bt^\circ$ mapping to any element of $\P^\circ$.

By Lemma~\ref{Exponential decay}, the fraction of elements $P$ of $\P$ with $\rho(P)\ge m$ is less than or equal to $4 \cdot 2^{-m/4}$.  There exists a subexponential sequence 
$\sigma_1,\sigma_2,\ldots$
such that 
$$|f^{-1}(P)|,\,|g^{-1}(P)|\le \sigma_m$$
if $\rho(P)\le m$,
so there exists $N$ for which
conditions (5) and (6) of Proposition~\ref{Master} hold.  Each element in $\At\setminus \At^\circ$ either has $\rho$-invariant greater than $m$ or has invariant $\le m$ but satisfies 
$P(x)\equiv P(\zeta x)$ for some $\zeta \neq 1$.  If $m$ is sufficiently large in absolute terms, we may assume that the contribution of all elements of with $\rho$-invariant greater than $m$ to
either $\At$ or $\Bt$ represents less than a $\delta/2$ fraction of the total elements of $\At$ or $\Bt$ respectively.  Once $m$ is fixed, we have an upper bound for the size of fibers of $f$ or $g$, so if $q^n$ is sufficiently large, Lemma~\ref{A orbits} implies that the contribution of all fibers of all elements of $P$ with $P(x)\equiv P(\zeta x)$, as $\zeta$ ranges over all elements other than $1$,
is again less than a $\delta/2$ fraction of the elements of $\At$ or $\Bt$.  To summarize, we have proven the following.

\begin{prop}
\label{1 to 6}
If $G$ is sufficiently large, for all $\delta >0$ if $m$ is chosen to be sufficiently large, conditions (1)--(6) of Proposition~\ref{Master} hold for $\P^\circ$ defined by $m$-regularity.
\end{prop}

\section{End of the proof}

\begin{prop}
\label{Few large centralizers}
For all $\epsilon > 0$, there exists $C$ such that if $G$ is a classical finite simple group of rank $r$, the fraction of elements $P\in \P$ such that some element of $f^{-1}(P)$
has centralizer order greater than $C q^r$ in $\tilde G$ is less than $\epsilon$.
\end{prop}

\begin{proof}
By Lemma~\ref{Exponential decay}, if $q > 8/\epsilon$, the fraction of elements $P\in\P$ for which $\rho(P)>0$ is less than $\epsilon/2$.  When $\rho(P)=0$, 
the centralizer of every element of $f^{-1}(P)$ is the group of $\F_q$-points of a maximal torus.  We claim that this group has order $\le \alpha(P)q^r$.
For semisimple $s\in \SL_{r+1}(\F_q)$ with characteristic polynomial $Q_1\cdots Q_j$, $Q_i$ irreducible, the order of the centralizer of $s$ is
$$\frac{\prod_{i=1}^j (q^{\deg Q_i}-1)}{q-1} \le (1-q^{-1})^{-1} q^r \le \alpha(P)q^r.$$
For semisimple $s\in \SU_{r+1}(\F_q)$ with characteristic polynomial $Q_1\cdots Q_j$, the order of the centralizer is
$$\frac{\prod_{i=1}^j (q^{\deg Q_i}-(-1)^{\deg Q_i})}{q+1} \le \alpha(P)q^r.$$
For $\SO_{2r+1}(\F_q)$, $\Sp_{2r}(\F_q)$ or $\SO^{\pm}_{2r}(\F_q)$ every irreducible $Q_i=Q_i^*$ of degree $\ge 2$ contributes a factor of $q^{\deg Q_i/2}-1$, while every pair $\{Q_i,Q_j\}$ with $Q_j = Q_i^*$
contributes a factor of $q^{\deg Q_i}+1 = q^{\deg Q_j}+1$, so the centralizer order is less than $\alpha(P) q^r$.

By Lemma~\ref{Undersized classes are rare}, if $q$ is sufficiently large, we may
assume $\alpha(P) < 2$ for all but an $\epsilon/2$ fraction of elements of $\P$, and the lemma follows.  It therefore suffices to prove the lemma when $q$ is fixed.
By Lemma~\ref{Exponential decay}, we may additionally assume $\rho(P)$ is bounded.   We can therefore factor $P$ as a product of two polynomials,
a square-free factor $Q$ and a factor $R$ relatively prime to $Q$ and of bounded degree.  The centralizer of $s$ is therefore a product of a torus with $\le \alpha(Q) q^{r-r_0}$
elements, and a connected reductive group of rank $r_0$, with a bounded number of elements.  This gives the desired bound.

\end{proof}

The following theorem is not needed for the main result but may be of interest in its own right.

\begin{thm}
\label{CC to elements}

For all $\epsilon > 0$ there exists $\delta > 0$ such that if $G$ is a finite simple group of Lie type, and $S$ is a normal subset of $G$ with less than $\delta |G|$ elements, then
$S$ consists of less than $\epsilon k(G)$ conjugacy classes.  
\end{thm}

\begin{proof}

First we assume that $G$ is of type A--D and of sufficiently high rank, so we are in the setting of Proposition~\ref{Master}.
Let $T\subset \A$ denote the set of conjugacy classes in $G$ corresponding to elements of $S$.  
By construction, the cardinality of any fiber of $\phi\colon \At\to \A$ over $\A^\circ$ is $|Z|$.
By Proposition~\ref{1 to 6}, hypotheses (1) and (3) of Proposition~\ref{Master} hold when $\At = {\tilde G}^{\tilde G}$
and $\A = G^G$.  Choosing the parameter $\delta$ of this proposition sufficiently small, we may assume
that $|\phi^{-1}(T)| > |Z| |T|/2$.  

Let $\tilde S$ denote the inverse image of $S$ in $\tilde G$, so $|\tilde S|/|\tilde G| = |S|/|G|$.
Let $\tilde T\subset \At$ denote the set of conjugacy classes in $\tilde G$ corresponding to elements of $\tilde S$.  Then $\tilde T = \phi^{-1}(T)$,
so 
$$\frac{|\tilde T|}{|\At|} > \frac{ |Z| |T|}{2|\At|} \ge \frac{|T|}{2|\A|}.$$
It therefore suffices to prove that for all $\epsilon > 0$ there exists $\delta > 0$ such that for all normal subsets $\tilde S$ of $\tilde G$ with 
$|\tilde S| \le \delta |\tilde G|$ elements, the number $|\tilde T|$ of conjugacy classes in $\tilde S$ is $\le \epsilon|\At| = \epsilon k(\tilde G)$.  

By Proposition~\ref{Few large centralizers}, fixing $C$ sufficiently large, the fraction of elements in $\P$
whose fibers have elements with centralizer order greater than $C q^r$ is as small as desired.  By inequality \eqref{Few A-tildes in big fibers} and condition (7), 
the fraction of elements in $\tilde G$
with centralizer order greater than $C q^r$ is likewise as small as desired.  Therefore, we may assume that in any normal subset $\tilde S$ of $\tilde G$ containing $\epsilon k(\tilde G)$
conjugacy classes, at least $\epsilon k(\tilde G)/2$ have centralizer order $\le C q^r$.  These account for at least 
$$\frac{\epsilon k(\tilde G) |\tilde G|}{2 C q^r}$$
elements in $\tilde S$.  By \cite[Theorem~1.1]{FG}, $k(\tilde G) \ge q^r$, so we may take $\delta := \epsilon /2C$.
This leaves the bounded rank case.  

In the limit as $q\to \infty$, the fraction of elements of $\tilde G$ which are regular semisimple goes to $1$.  The centralizers of any semisimple element is connected reductive, and for regular semisimple elements, the centralizer is a torus and therefore has at most $(q+1)^r$ elements.  As $r$ is fixed, this gives an upper bound of the form $Cq^r$.  Thus, the above claim for $\tilde S$ holds as in the high rank classical case.

Again, in the large $q$ limit, the fraction of elements $g$ of $\tilde G$ which are conjugate to $gz$ for some non-trivial central element $z\in Z$ goes to $0$, by the uniform version of the Lang-Weil estimate.  In the complement of this set of elements, the map from conjugacy classes of $\tilde G$ to conjugacy classes of $G$ is $|Z|$ to $1$.  Thus, the theorem for $S$ again reduces to the claim for $\tilde S$, just as in the classical case.

For Suzuki and Ree groups, the argument given above must be modified slightly, since the Lang-Weil estimate does not apply to the number of points in such a group lying on a subvariety of the ambient algebraic group, but one can use \cite[Th.~4.2]{LP} as a replacement.
\end{proof}

We now prove Theorem~\ref{main}.

\begin{proof}
We need only prove that given $\epsilon > 0$, if $|G|$ is sufficiently large, we can choose  $\delta > 0$ and then $m$ so that 
for $\P^\circ$ defined by $m$-regularity,
condition (9) of Proposition~\ref{Master} holds.

For all $k>0$ there exists $N$ such if $n\geq 2$ and $q$ is a prime power, then the fraction of elements of $\L_n(q)$, $\U_n(q)$, or $\O_{n}(q)$
with more than $N\log n$ factors is less than $n^{-k}$ \cite[Prop.~2.4--2.6]{LS}.
Therefore, for every $k$, for sufficiently large $n$, in any of these groups, the fraction of elements with no irreducible factor
of degree $>2\sqrt n$ is less than $n^{-k}$, which can be taken as small as we wish.
In the case that $G$ is of linear or unitary type, by Lemma~\ref{polynomial estimate},
we may further assume that no prime factor of $n$ is $\equiv 1$ modulo the degree of an irreducible factor with degree $>2\sqrt n$.
Assuming $P_1\in \P^\circ$ has an irreducible factor $Q$ of degree $> 2\sqrt n$,
then by Lemma~\ref{polynomial estimate}, the fraction of elements $P_2\in\P^\circ$ such that 
$P_2$ has an irreducible factor whose degree is an integer multiple of $\deg Q/2$ goes to $0$ as $n\to \infty$.  
By Lemma~\ref{Just one}, if $\epsilon_1>0$ and $n$ is sufficiently large, at least a $1-\epsilon_1$ fraction of pairs $(P_1,P_2)\in\P^2$ have the property that if $s_i\in \tilde G$ is semisimple and maps to $P_i$ for $i=1$ and $i=2$, then there exists a prime $\ell$ such that
\begin{equation}
\label{Divisibility conditions}
\ell\mid |Z_{\tilde G}(s_1)|,\,\ell\nmid |Z_{\tilde G}(s_2)|.
\end{equation}
By construction, $\ell$ does not divide $n$.

To prove condition (9), we may partition the set 
$$\{(P_1,P_2)\in \P^\circ\times \P^\circ \mid(P_1,P_2)\in (f,g)((\phi,\psi)^{-1}(\X))\}$$
into two subsets, one consisting of pairs  satisfying \eqref{Divisibility conditions}
and one consisting of pairs which do not satisfy it.  We have already bounded the latter set, and it suffices to prove that for $\epsilon_2>0$, if $n$ is sufficiently large, the first set has less than 
$\epsilon_2|\P|^2$ elements.  Suppose that this is not the case.  

If $(P_1,P_2)$ belongs to the subset satisfying
\eqref{Divisibility conditions}, choose $\tilde g^{\tilde G}\in \At$ lying over $P_1$ and $(\chi,\chi_{\ad})$ in $\Bt$ lying over $P_2$
such that $\chi(g)\neq 0$, where $g^G = \phi(\tilde g^{\tilde G})$.
We denote the semisimple conjugacy class of $\tilde G = (\uG_{\ad})^*(\F_q)$ associated to $\chi_{\ad}$ by $t^{\tilde G}$, so the image of $t$ in its natural representation has characteristic polynomial $P_2$.
By Lemma~\ref{Which Z}, there exists a Zsigmondy prime $\ell>\sqrt{n}$ which divides the order of the centralizer of $g$ in $G$ but not the order of the centralizer of $t$ in $\tilde G$.
Therefore,
$$\ord_\ell |g^G| < \ord_\ell |G|\ = \ord_\ell |\tilde G| = \ord_\ell \chi_{\ad}(1).$$
The fraction $\chi_{\ad}(1)/\chi(1)$ is an integer dividing the order of the center of $\tilde G$
and therefore not divisible by $\ell$, so $\ell$ divides $d_{g^G,\chi}$.  As the map $\phi$ is at most $|Z|$ to $1$,
we obtain at least $\epsilon_2 |\P|^2/|Z|$ such pairs $(g^G,\chi)$.

For any  $\alpha\in \Q(\zeta_{|G|})$, 
let $T(\alpha)$ denote the normalized trace 
$$\frac{1}{[\Q(\zeta_{|G|}):\Q]}\Tr_{\Q(\zeta_{|G|})/\Q}(\alpha).$$
Note that $\Gal(\Q(\zeta_{|G|})/\Q)$ is commutative, and complex conjugation is an element of the group, so if the $\Aut(\C)$-orbit of a non-zero algebraic integer $\alpha$ is $\{\alpha_1,\ldots,\alpha_k\}$, then
$$T(|\alpha|^2) = T(\alpha\bar \alpha) = \frac 1k \sum_{i=1}^k \alpha_i\bar\alpha_i \ge 1.$$
As $(g^G,\chi)\in \X$, by definition $\chi(g)\neq 0$, so
$T(|\chi(g)|^2)\ge n$. 
Therefore,
$$\sum_{h\in g^G} T(|\chi(h)|^2) \ge n |g^G|.$$
If $n$ is sufficiently large, by Proposition~\ref{Few large centralizers}, we may assume that at least $\epsilon_3|\P|^2/|Z|$
pairs $(g^G,\chi)$ arising in this way satisfy $|g^G| > |G|/Cq^r$.  Thus,
\begin{align*} |G|\cdot \Irr(G)  &= T(\sum_{h\in G}\sum_{\chi\in \Irr(G)}|\chi(h)|^2) \\
&= \sum_{h\in G}\sum_{\chi\in \Irr(G)} T(|\chi(h)|^2) \ge \frac{\epsilon_3|\P|^2}{|Z|}\frac{n|G|}{Cq^r}\\
&= \frac{\epsilon_3q^rn|G|}{|Z|C}.
\end{align*}
By \cite[Th.~1.1]{FG}, $|\Irr(G)| \le 27.2q^r$.  For orthogonal and symplectic groups, $|Z|\le 4$, so that $|\Irr(G)| < 109 q^r/|Z|$.
By \cite[Cor.~3.7]{FG}, $|\Irr(G)| < 4 q^r/|Z|$ for $G = \PSL_{r+1}(q)$, and by \cite[Prop.~3.10]{FG},
$|\Irr(G)| \le 9 q^r/|Z|$ for $G = \PSL_{r+1}(q)$.  Putting these together, we deduce
$109 > \epsilon_3n/C$, which is impossible for large $n$.

\end{proof}

\end{document}